\documentclass[12pt,reqno]{amsart}
\usepackage[headings]{fullpage}
\usepackage{amssymb,amsmath,amscd,bbm,tikz}
\usepackage[all,cmtip]{xy}
\usepackage{url}
\usepackage[bookmarks=true,%
    colorlinks=true,%
    linkcolor=blue,%
    citecolor=blue,%
    filecolor=blue,%
    menucolor=blue,%
    urlcolor=blue,%
    breaklinks=true]{hyperref}
\usepackage{slashed}    
\usepackage{verbatim}
\usepackage{mathtools}
\usepackage[normalem]{ulem}
\graphicspath{{figures/}}

\newtheorem{theorem}{Theorem}[section]
\theoremstyle{definition}
\newtheorem{proposition}[theorem]{Proposition}

\newtheorem{remark}[theorem]{Remark}

\def\BZ{\mathbbm Z}

\def\BR{\mathbbm R}
\def\BC{\mathbbm C}

\def\calF{\mathcal F}

\def\calN{\mathcal N}

\def\calD{\mathcal D}

\def\calT{\mathcal T}

\def\ti{\widetilde}

\def\SL{\mathrm{SL}}

\def\pt{\partial}

\def\be{\begin{equation}}
\def\ee{\end{equation}}
\def\z{{\hat{z}}}

\def\rdet{\mathrm{det}^r_{\calN (\pi)}}

\newcommand{\rvline}{\hspace*{-\arraycolsep}\vline\hspace*{-\arraycolsep}}

\makeatletter

\renewcommand\thepart{\@Roman\c@part}%
\renewcommand\part{%
   \if@noskipsec \leavevmode \fi
   \par
   \addvspace{6.7ex}%
   \@afterindentfalse
   \secdef\@part\@spart}
\def\@part[#1]#2{%
    \ifnum \c@secnumdepth >\m@ne
      \refstepcounter{part}%
      \addcontentsline{toc}{part}{Part~\thepart.\ #1}%
    \else
      \addcontentsline{toc}{part}{#1}%
    \fi
    {\parindent \z@ \raggedright
     \interlinepenalty \@M
     \normalfont
     \ifnum \c@secnumdepth >\m@ne
       \centering\large\scshape \partname~\thepart.%
       \hspace{1ex}%
     \fi%
     \large\scshape #2%
     \markboth{}{}\par}%
    \nobreak
    \vskip 4.7ex
    \@afterheading}
  \def\@spart#1{
  \refstepcounter{part}%
  \addcontentsline{toc}{part}{#1}%
    {\parindent \z@ \raggedright
     \interlinepenalty \@M
     \normalfont
     \centering\large\scshape #1\par}%
     \nobreak
     \vskip 4.7ex
     \@afterheading}
\renewcommand*\l@part[2]{%
  \ifnum \c@tocdepth >-2\relax
    \addpenalty\@secpenalty
    \addvspace{0.75em \@plus\p@}%
    \begingroup
      \parindent \z@ \rightskip \@pnumwidth
      \parfillskip -\@pnumwidth
      {\leavevmode
       \normalsize \bfseries #1\hfil \hb@xt@\@pnumwidth{\hss #2}}\par
       \nobreak
       \if@compatibility
         \global\@nobreaktrue
         \everypar{\global\@nobreakfalse\everypar{}}%
      \fi
    \endgroup
  \fi}

\def\l@subsection{\@tocline{2}{0pt}{2pc}{6pc}{}}
\makeatother

\begin{document}
\title[The (twisted/$L^2$)-Alexander polynomial of ideal triangulations]{
  The (twisted/$L^2$)-Alexander polynomial of ideally 
triangulated 3-manifolds}
\author{Stavros Garoufalidis}
\address{
International Center for Mathematics, Department of Mathematics \\
Southern University of Science and Technology \\
Shenzhen, China
\newline
{\tt \url{http://people.mpim-bonn.mpg.de/stavros}}}
\email{stavros@mpim-bonn.mpg.de}

\author{Seokbeom Yoon}
\address{International Center for Mathematics \\
Southern University of Science and Technology \\
Shenzhen, China \newline
{\tt \url{http://sites.google.com/view/seokbeom}}}
\email{sbyoon15@gmail.com}



\date{3 January 2024}

\begin{abstract}
  We establish a connection between the Alexander polynomial of a knot and its
  twisted and $L^2$-versions with the triangulations that appear in
  3-dimensional hyperbolic geometry. Specifically, we introduce twisted
  Neumann--Zagier matrices of ordered ideal triangulations  and use them to provide
  formulas for the Alexander polynomial and its variants, the twisted Alexander
  polynomial and the $L^2$-Alexander torsion.
\end{abstract}

\maketitle

{\footnotesize
\tableofcontents
}


\section{Introduction}
\label{sec.intro}

The Alexander polynomial is a fundamental invariant of knots that dates back to
the origins of algebraic topology \cite{Alexander}. It has been studied time and
again from various points of view that include  twisting by a representation
\cite{Wada94, Lin2001}, or considering $L^2$-versions \cite{Luck2002, DFL15}. 
There are numerous results and surveys to this subject that the reader may consult
that include~\cite{friedl2011survey, DFJ2012, Kitano2015}. 

The goal of the paper is to establish a connection of this classical topological
invariant and its variants with the triangulations that appear in  3-dimensional
hyperbolic geometry \cite{Thurston}.   
These triangulations involve ideal tetrahedra, which one can think of as tetrahedra
with their vertices removed, whose faces are identified in pairs so as to obtain
the interior of a compact 3-manifold. Under such an identification, an edge can lie in
more than one tetrahedron, or said differently, going around an edge, one traverses
several tetrahedra, possibly with repetition. Keeping track of the total number a
tetrahedron winds around an edge (in each of its possible three ways) gives rise to a
pair of Neumann--Zagier matrices~\cite{NZ}. An ideal triangulation of such a manifold
lifts to an ideal triangulation of its universal cover, and this gives rise to twisted
Neumann--Zagier matrices; see Section~\ref{sec.nz} for details. 

Our main theorems provide explicit relations of the twisted Neumann--Zagier matrices
with the Alexander polynomial and its variants, the twisted Alexander polynomial and 
the $L^2$-Alexander torsion (Theorems~\ref{thm.1}--\ref{thm.3}). These relations
follow from a connection between the twisted Neumann-Zagier and Fox calculus
\cite{Fox1} which we will discuss in Section~\ref{sec.fox}.

The paper is organized as follows. 
In Section~\ref{sec.nz}, we briefly recall Neumann--Zagier matrices and introduce
their twisted version. In Section~\ref{sec.det}, we present our main theorems and their
corollaries. In Section~\ref{sec.fox}, we show that twisted Neumann--Zagier matrices
can be obtained from Fox calculus and prove our main theorems. We give an explicit
computation for the figure-eight knot and verify our theorems in 
Section~\ref{sec.example}.

\section{Twisted Neumann--Zagier matrices}
\label{sec.nz}

In this section we briefly recall ideal triangulations of 3-manifolds, 
their gluing equation and Neumann--Zagier matrices following~\cite{Thurston,NZ},
and introduce their twisted versions. 
Fix a compact 3-manifold $M$ with torus boundary and $\calT$ an ideal
triangulation of the interior of $M$. We denote the edges and the tetrahedra of
$\calT$ by $e_i$ and by $\Delta_j$, respectively, for $1 \leq i,j \leq N$. Note that
the number of edges is equal to that of tetrahedra. Every tetrahedron $\Delta_j$ is
equipped with shape parameters, i.e. each edge of $\Delta_j$ is assigned to one
shape parameter among $z_j, z_j'$ and $z_j''$ with opposite edges having same
parameters as in Figure~\ref{fig.tetrahedron}. If $\calT$ is \emph{ordered}, i.e. 
if every tetrahedron has vertices labeled with $\{0,1,2,3\}$ and every face-pairing
respects the vertex-order, then we assign the edges $(01)$  and $(23)$ of each
tetrahedron $\Delta_j$ to the shape parameter $z_j$.  

\begin{figure}[htpb!]
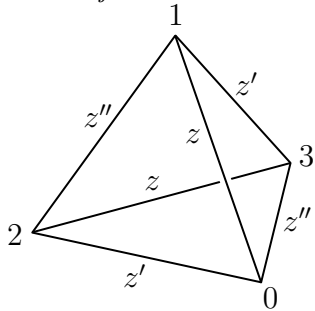
\caption{A tetrahedron with shape parameters.}
\label{fig.tetrahedron}
\end{figure}

The \emph{gluing equation matrices} $G, G'$ and $G''$ of $\calT$ are $N \times N$
integer matrices whose rows and columns are indexed by the edges and by the
tetrahedra of $\calT$, respectively. The $(i,j)$-entry of $G^\square$ for 
$\square \in \{ \ , ' , ''\}$ is the number of edges of $\Delta_j$ assigned to the
shape parameter $z_j^\square$ and identified with the edge $e_i$ in $\calT$. 
The \emph{Neumann--Zagier matrices} of $\calT$ are  defined as the differences of
the gluing equation matrices:
\be
A:=G-G' , \qquad  
B:=G''-G' \in M_{N \times N} ( \BZ  ) \, .
\ee

We now define a twisted version of the above matrices. These are essentially 
the Neumann--Zagier matrices of the ideal triangulation $\ti \calT$ of the universal
cover of $M$ obtained by pulling back $\calT$. We choose a lift $\ti e_i$ of $e_i$
and $\ti \Delta_j$ of $\Delta_j$ for all  $1 \leq i, j \leq N$ so that every edge and
tetrahedron of $\ti \calT$ is expressed as $\gamma \cdot \ti e_i$ or
$\gamma \cdot \ti \Delta_j$ for $\gamma \in \pi:=\pi_1(M)$.
Analogous to the gluing equation matrices, let $G_\gamma^\square$ for
$\square \in \{ \ , ' , ''\}$ and $\gamma \in \pi$ denote $N\times N$ integer matrices 
whose $(i,j)$-entry is the number of edges of $\gamma \cdot \ti \Delta_j$ assigned
to the shape parameter $z_j^\square$ and identified with the edge $\ti e_i$ in
$\ti \calT$.  We define the \emph{twisted gluing equation matrices} of $\calT$ by
\be
\label{eqn.Guniv}
\mathbf{G}^\square := 
\sum_{\gamma \in \pi}  G^\square_\gamma \otimes \gamma 
\in 
M_{N \times N}( \BZ [\pi] ) 
\ee
and the \emph{twisted Neumann--Zagier matrices} of $\calT$ by
\be
\mathbf{A} :=\mathbf{G} - \mathbf{G}', \qquad
\mathbf{B} :=\mathbf{G}'' - \mathbf{G}'
\in M_{N \times N} ( \BZ [\pi] ) \,.
\ee
The above notation differs slightly from the one used in~\cite{GY:twistedNZ};
hopefully this will not cause any confusion.  
Note that $G^\square_\gamma$ is the zero matrix for all but finitely many
$\gamma$, hence the sum in \eqref{eqn.Guniv} is finite. Since the above matrices
are well-defined after fixing lifts of each edge and tetrahedron of $\calT$,
a different choice of lifts changes $\mathbf{G}^\square$, $\mathbf{A}$
and $\mathbf{B}$ by multiplication from the left or right by the same diagonal
matrix with entries in $\pi$. This ambiguity propagates to any invariant constructed
using these matrices.

The Neumann--Zagier matrices of an ideal triangulation satisfy a key symplectic
property~\cite{NZ} which has been the source of many invariants in quantum topology.
In particular, it follows that $A B^T$ is a symmetric matrix. This property generalizes
for twisted Neumann--Zagier matrices
\be
\label{ABt}
\mathbf{A} \, \mathbf{B}^\ast = 
\mathbf{B} \, \mathbf{A}^\ast
\ee
where the adjoint $X^\ast$ of a matrix $X \in M_{N \times N}(\BZ[\pi])$ is given by
the transpose followed by the involution of $\BZ[\pi]$ defined by
$\gamma \mapsto \gamma^{-1}$ for all $\gamma \in \pi$. The above equation can be
proved by repeating the same argument as in the proof of 
\cite[Theorem 1.2]{GY:twistedNZ} or \cite{choi06}.

One important aspect of our results is the use of ordered ideal triangulations. 
It is known that every 3-manifold with nonempty
boundary has such a triangulation~\cite{BP:branched}. The choice of an ordered
triangulation breaks the symmetry between the two Neumann--Zagier matrices, and
distinguishes the $\mathbf{B}$ among the two. 


\section{Alexander invariants from twisted NZ matrices}
\label{sec.det}

In this section we express the Alexander polynomial and its twisted and $L^2$-versions
in terms of the twisted Neumann--Zagier matrix $\mathbf{B}$.
Throughout the section, we fix 
\begin{itemize}
\item[$(\dagger)$] a compact 3-manifold $M$ with torus boundary,
an \emph{ordered} ideal triangulation $\calT$ of the interior of $M$ and 
a group homomorphism $\alpha : \pi \rightarrow \BZ$. 
\end{itemize}

\subsection{Alexander polynomial}

The homomorphism $\alpha$ in $(\dagger)$ gives rise to a homomorphism 
$\alpha : \BZ[\pi] \rightarrow \BZ[\BZ] \simeq \BZ[t^{\pm1}]$ of group rings,
and we define
\be
\label{ABa}
\mathbf{A}_\alpha(t):= \alpha(\mathbf{A}), \qquad
\mathbf{B}_\alpha(t) := \alpha(\mathbf{B}) \in M_{N \times N}(\BZ[t^{\pm1}]) \,.
\ee
Our first theorem relates the determinant of one of these matrices
with the Alexander polynomial $\Delta_\alpha(t)$ associated with $\alpha$,
assuming that this is well-defined, that is, the (cellular) chain complex of $M$
with local coefficient twisted by $\alpha$ is acyclic. A typical case is $M$ being the
complement of a knot in a homology sphere with $\alpha$ being
the abelianization map. Note that the determinants of $\mathbf{A}_\alpha(t)$ and 
$\mathbf{B}_\alpha(t)$ as well as $\Delta_\alpha(t)$ are  well-defined up to 
multiplication by $\pm t^k$, $k \in \BZ$. Below, we denote by $\doteq$ the
equality of Laurent polynomials (or functions of $t$) up to multiplication by
$\pm t^k$, $k \in \BZ$.

\begin{theorem} 
\label{thm.1}
Fix $M$,$\calT$ and $\alpha$ as in $(\dagger)$.  
Then either $\det \mathbf{B}_\alpha(t)=0$ or
\be
\label{eqn.Bsim}
\det \mathbf{B}_\alpha(t) \doteq  \frac{\Delta_\alpha(t)}{t-1} \, (t^n-1)^{m}
\ee
for some $n \geq 0$ and $m \geq 1$.	
\end{theorem}

The matrix $\mathbf{A}_\alpha(t)$ also satisfies a similar equation, but only
modulo 2. See Remark~\ref{rmk.A} for details. 

\subsection{Twisted Alexander polynomial}

The homomorphism $\alpha$ in  Theorem~\ref{thm.1} can be replaced by 
$\alpha \otimes \rho$ for any representation $\rho:\pi \rightarrow \SL_n(\BC)$,
provided that the twisted Alexander polynomial $\Delta_{\alpha \otimes \rho}(t)$
associated with $\alpha \otimes \rho$ is defined. This happens when the (cellular)
chain complex of $M$ with local coefficient twisted by $\alpha \otimes \rho$ is
acyclic. A typical case is $M$ being the complement of a hyperbolic knot in a
homology sphere with $\rho : \pi \rightarrow \SL_2(\BC)$ being a lift of the geometric
representation.

\begin{theorem} 
\label{thm.2}
Fix $M$,$\calT$ and $\alpha$ as in $(\dagger)$ 
and  a representation $\rho : \pi \rightarrow \mathrm{SL}_n(\BC)$.
Then either $\det \mathbf{B}_{\alpha \otimes \rho} (t) = 0$ or
\begin{align}
\label{eqn.B3}
\det \mathbf{B}_{\alpha \otimes \rho} (t) & \doteq \Delta_{\alpha \otimes \rho}(t) \,
\det (\rho (\gamma)\, t^ {\alpha(\gamma)} - I_n)^{m}
\end{align}
for some peripheral curve $\gamma$ and $m \geq 1$ 
where $I_n$ is the identity matrix of rank $n$.
\end{theorem}

Note that if $\rho$ is the trivial 1-dimensional representation, we have 
$\mathbf{B}_{\alpha \otimes \rho}(t) = \mathbf{B}_\alpha(t)$ and 
$\Delta_\alpha(t)/(t-1) = \Delta_{\alpha \otimes \rho}(t)$ \cite{Wada94}.
Hence Theorem~\ref{thm.1} is a special case of Theorem~\ref{thm.2}.

\subsection{$L^2$-Alexander torsion}

In~\cite{DFL15} Dubois--Friedl--L\"{u}ck introduced the $L^2$-Alexander
 torsion as an $L^2$-version of the Alexander polynomial
\be
\tau^{(2)}(M,\alpha) : \BR^{+} \rightarrow [0,\infty), \qquad
t \mapsto  \tau^{(2)}(M,\alpha)(t)\,.
\ee
As the Alexander polynomial, $\tau^{(2)}(M,\alpha)$ is well-defined 
up to  multiplication by a function $t \mapsto t^r$ for $r \in \BR$. We will write 
$f \doteq g$ for functions $f$ and $g : \BR^{+} \rightarrow [0,\infty)$ if 
$f(t) = t^r g(t)$ for some $r \in \BR$. Briefly, for fixed $t>0$, 
$\tau^{(2)}(M,\alpha)(t)$ is defined to be the $L^2$-torsion of the chain 
complex of $\BR[\pi]$-modules
\be
\BR[\pi] \otimes_{\BZ[\pi]} C_\ast(\widetilde{M};\BZ)
\ee 
where $\ti M$ is the universal cover of $M$ and $\BR[\pi]$ is viewed as a
$\BZ[\pi]$-module using the homomorphism
\be
\label{eqn.at}
\alpha_t : \BZ[\pi] \rightarrow \BR[\pi], \quad g \mapsto t^{\alpha(g)} g \,.
\ee
The $L^2$-torsion of the above complex is defined in terms of the Fulgede-Kadison
determinant of matrices with entries in $\BR[\pi]$.
Roughly speaking, the Fulgede-Kadison determinant of a matrix $X$ is defined in
terms of the spectral density function of $X$, viewed as a map between direct sums
of the Hilbert space $\ell^2 (\pi)$ of squared-summable formal sums over $\pi$. We
refer to \cite{Luck2002, DFL15} for the  precise definition. However, we will not use
the definition, but only some basic properties for square matrices, such as
\be
\label{FKprops}
\begin{aligned}
\rdet(XY)&=\rdet(X)\, \rdet(Y) \, , \\
\rdet \begin{pmatrix}  X & 0 \\ 
Z& Y
\end{pmatrix} &=\rdet(X)\, \rdet (Y)  \, .
\end{aligned}
\ee
Here $X$ and $Y$ are square matrices with entries in $\BR[\pi]$, and $\rdet(X)$
denotes the regular Fuglede-Kadison determinant of $X$,  which equals to the
Fuglede-Kadison determinant of $X$ if $X$ has full rank, and zero otherwise.

We now consider the Fuglede--Kadison determinant of the twisted Neumann--Zagier
matrices and relate it with the $L^2$-Alexander torsion. 
Recall that the twisted Neumann--Zagier matrices are square matrices with entries
in the group ring $\BZ[\pi]$. We define a function
\be
\mathrm{det} (\mathbf{B},\alpha) :  \BR^{+} \rightarrow [0,\infty),\qquad
t \mapsto \rdet (\alpha_t(\mathbf{B}))
\ee
where $\alpha_t : \BZ[\pi] \rightarrow \BR[\pi]$ is the homomorphism given
in~\eqref{eqn.at}.

\begin{theorem}
\label{thm.3}
Fix $M$,$\calT$ and $\alpha$ as in $(\dagger)$.
Suppose that every component of the $Z$-curves of $\calT$ has infinite
order in $\pi$ (see Section~\ref{sub.curves} for the definition of $Z$-curves). 
Then we have
\be
\mathrm{det} (\mathbf{B},\alpha) \doteq  \tau^{(2)}(M,\alpha)
\, \mathrm{max} \{1, t^{n}\}
\ee
for some $n \in \BZ$.
\end{theorem}


\section{Fox calculus and twisted NZ matrices}
\label{sec.fox}

In this section, we discuss a connection between twisted Neumann--Zagier matrices 
and Fox calculus, and prove Theorems~\ref{thm.1}--\ref{thm.3}. 

\subsection{Fox calculus}

Let $M$ be a compact 3-manifold with torus boundary and $\calT$ an ideal
triangulation of the interior of $M$ with $N$ tetrahedra. The dual complex
$\calD$ of $\calT$ is a 2-dimensional cell complex with $2N$ edges and $N$ faces.
We choose an orientation of each edge and let $\calF_\calD$ be the free group
generated by the edges of $\calD$; if $\calT$ is ordered, we choose the orientation
by the one induced from the vertex-order. 

The faces of $\calD$ correspond to words $r_1,\ldots,r_N \in \calF_\calD$
well-defined up to conjugation. Two consecutive letters of $r_i$ ($1 \leq i \leq N$) 
correspond to two adjacent face pairings of $\calT$,
hence there is a shape parameter lying in between. Here we regard that the first
and the last letter of $r_i$ are also consecutive. Inserting such shape parameters
between the letters of $r_i$, we obtain a word $R_i$ whose length is two times that
of $r_i$. More precisely, let $\calF_\z$ be the free group generated by 
$\z^\square_j$ for $1 \leq j \leq N$ and $\square \in \{ \ , ' , ''\}$ (hence
$\calF_\z$ has
$3N$ generators) where $\z^\square_j$ is a formal variable corresponding to a shape
parameter $z^\square_j$. Then we define a word $R_i \in \calF_\calD \ast \calF_\z$ by
its $2k$-th letter to be the $k$-th letter of $r_i$
and its $(2k-1)$-st letter to be a generator of $\calF_\z$ corresponding to
the shape parameter lying between the $(k-1)$-st and the $k$-th letters of $r_i$.
Here $k \geq 1$ and the $0$-th letter of $r_i$  means the last letter of $r_i$.

We choose $N-1$ generators of $\calF_\calD$  forming a spanning tree in $\calD$ 
and define a map 
\begin{equation}
p : \calF_\calD \ast \calF_\z \rightarrow \pi
\end{equation} 
by eliminating those $N-1$ generators of $\calF_\calD$ and all generators
$\z^\square_j$ of $\calF_\z$. Note that the rest $N+1$ generators of 
$\calF_\calD$ with $N$ relators  $p(r_1), \ldots, p(r_N)$ give a presentation
of $\pi = \pi_1(M)$, hence the map $p$ is well-defined.

\begin{proposition} 
\label{prop.fox}
The twisted gluing equation matrices $\mathbf{G}^\square $ of $\calT$
agree with
\be
\label{eqn.matrix}
\begin{pmatrix}
p\left(\dfrac{\pt R_1}{\pt \z^\square_{1}}\right) & \cdots  &
p\left(\dfrac{\pt R_1}{\pt \z^\square_{N}}\right) \\
\vdots & & \vdots \\
p \left(\dfrac{\pt R_N}{\pt \z^\square_1}\right) & \cdots  &
p\left(\dfrac{\pt R_N}{\pt \z^\square_{N}}\right)
\end{pmatrix} \in M_{N \times N}(\BZ[\pi])
\ee
up to left multiplication by a diagonal matrix with entries in $\pi$.
\end{proposition}

\begin{proof} 
Let $\ti \calT$ be the ideal triangulation of the universal cover of $M$ induced from
$\calT$. For two tetahedra $\Delta$ and $\Delta'$ of $\ti \calT$ let
$d(\Delta, \Delta') \in \calF_\calD$ be a word representing an oriented curve that
starts at $\Delta$ and ends at $\Delta'$. We choose a lift $\ti \Delta_j$ of each
tetrahedron $\Delta_j$ of $\calT$ such that 
\be
\label{eqn.lift}
p\left(d(\ti \Delta_{j_0},\ti \Delta_{j_1})\right) =1 
\ee
for all $1 \leq j_0, j_1 \leq N$.
We also choose any lift $\ti e_i$ of each edge $e_i$ of $\calT$ so that 
the twisted gluing equation matrices $\mathbf{G}^\square$ are
determined. Precisely, the $(i,j)$-entry of $\mathbf{G}^\square$ is given by 
\be
\label{eqn.pt1}
\sum_\Delta p \left(d(\ti \Delta_1,\Delta) \right) \in \BZ[\pi]
\ee
where the sum is taken over all tetrahedra $\Delta$ of $\ti \calT$ contributing
$z_j^\square$ to $\ti e_i$. The index of $\ti \Delta_1$ can be replaced
by any $1 \leq j \leq N$ due to Equation~\eqref{eqn.lift}.

On the other hand, there is an initial tetrahedron, say  $\hat \Delta_{i}$, around 
$\ti e_i$ such that the word $r_i \in \calF_\calD$ is obtained by winding around the
edge $\ti e_i$ starting from $\hat \Delta_{i}$. Then it follows from the definition
of $R_i$ that 
\be
\label{eqn.pt2}
p\left(\frac{\pt R_i}{\pt \z_j^\square} \right)
= \sum_\Delta p \left( d(\hat \Delta_{i},\Delta) \right) \in \BZ[\pi]
\ee
where the sum is taken over all tetrahedra $\Delta$ of $\ti \calT$ contributing
$z_j^\square$ to $\ti e_i$. Since 
\be
p\left(d(\ti \Delta_1, \Delta)\right)
= p\left(d(\ti \Delta_1, \hat \Delta_{i}) \right) \,
p\left( d(\hat \Delta_{i}, \Delta)\right)
\ee 
for any $\Delta$, we deduce from~\eqref{eqn.pt1} and~\eqref{eqn.pt2} that the
matrix~\eqref{eqn.matrix} agrees with $\mathbf{G}^\square$ up to left
multiplication by a diagonal matrix with entries in $\pi$.
\end{proof}

\subsection{Curves in triangulations}
\label{sub.curves}

The 1-skeleton $\calD^{(1)}$ of the dual complex $\calD$ intersects with a 
tetrahedron in four points. Hence there are three ways of smoothing it in each 
tetrahedron as in Figure~\ref{fig.smoothing_abc}.  Each smoothing makes
two curves in a tetrahedron winding two edges with the same shape parameter. 
We thus refer to it as $Z$, $Z'$, or $Z''$-smoothing accordingly. 
Applying $Z$-smoothing to $\calD^{(1)}$ for all tetrahedra, we obtain finitely
many loops, which we call $Z$-curves of $\calT$.
We define $Z'$ and $Z''$-curves of $\calT$ similarly.

\begin{figure}[htpb!]
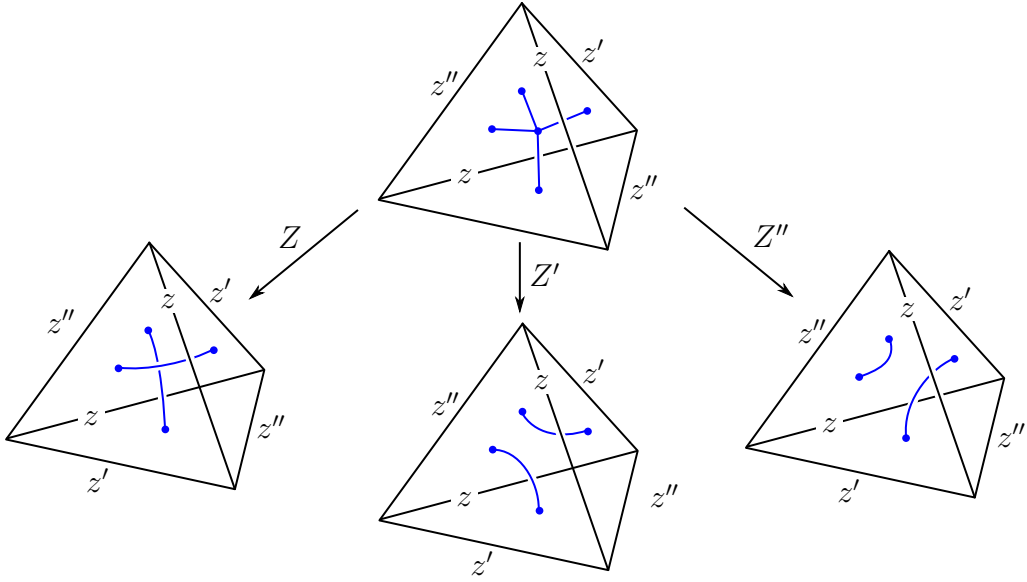
\caption{Three ways of smoothing $\calD^{(1)}$.}
\label{fig.smoothing_abc}
\end{figure}
 
\begin{proposition}
\label{prop.Bloop}
If $\calT$ is ordered, the $Z$-curves homotope to disjoint peripheral curves.
\end{proposition}

\begin{proof}
For ordered $\calT$,  each face of $\calT$ has a ``middle'' vertex, the one whose
label is neither greatest nor smallest among the three vertices of the face. Recall
that the $Z$-curves intersect with each face $f$ of $\calT$ in a point. We  push the
intersection point toward the middle vertex of $f$. Doing so for all faces of $\calT$,
the $Z$-curves homotope to disjoint peripheral curves. Note that then the
$Z$-curves make two small curves in each tetrahedron lying in a neighborhood of
the vertices 1 and 2 as in Figure~\ref{fig.push}.
\begin{figure}[htpb!]
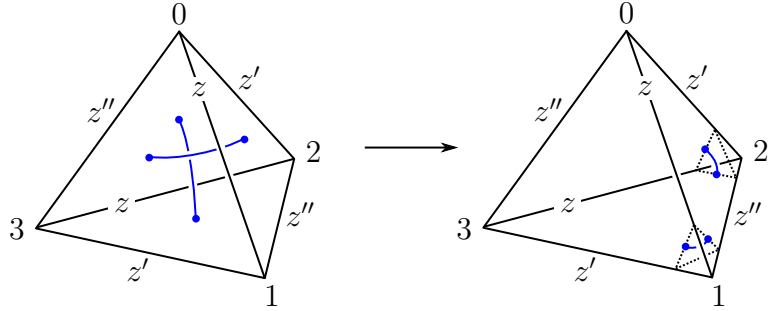
\caption{Homotope $Z$-curves to peripheral curves.}
\label{fig.push}
\end{figure}
\end{proof}

We now fix a tetrahedron $\Delta_j$ of $\calT$. Recall that the free group
$\calF_\calD$ has $2N$ generators, say $g_1,\ldots, g_{2N}$, and 
that a face $f$ of $\Delta_j$ corresponds to one generator $g_i$,
oriented either inward or outward to $\Delta_j$. 
We define a column vector  $v_f \in \BZ[\pi]^{2N}$
\be
\label{eqn.v}
v_f = 
\begin{cases}
p(g_i)\, e_i & \quad \textrm{if $g_i$ is inward to $\Delta_j$} \\
-e_i &  \quad  \textrm{if $g_i$ is outward to $\Delta_j$}
\end{cases}
\ee
where $(e_1,\ldots,e_{2N})$ is the standard basis of $\BZ^{2N}$. 
We say that two faces of $\Delta_j$ are \emph{$Z$-adjacent} if they are joined by
one of two curves in $\Delta_j$ obtained from $Z$-smoothing
(see~Figure~\ref{fig.smoothing_abc}). 
Note that $\Delta_j$ has  two pairs of $Z$-adjacent faces.

\begin{proposition}
\label{prop.vec}
If $\calT$ is ordered, the column vector
\be
\label{eqn.vec}
\begin{pmatrix}
p\left(\dfrac{\pt r_1}{\pt g_{1}}\right) & \cdots  &
p\left(\dfrac{\pt r_1}{\pt g_{2N}}\right) \\
\vdots & & \vdots \\
p\left(\dfrac{\pt r_N}{\pt g_1}\right) & \cdots  &
p\left(\dfrac{\pt r_N}{\pt g_{2N}}\right)
\end{pmatrix} (v_{f_0}+v_{f_1}) \in \BZ[\pi]^N
\ee
is equal to the $j$-th column of
\be
\label{eqn.vecb}
\begin{pmatrix}
p\left(\dfrac{\pt R_1}{\pt \z''_{1}}\right) & \cdots  &
p\left(\dfrac{\pt R_1}{\pt \z''_{N}}\right) \\
\vdots & & \vdots \\
p\left(\dfrac{\pt R_N}{\pt \z''_1}\right) & \cdots  &
p\left(\dfrac{\pt R_N}{\pt \z''_{N}}\right) 
\end{pmatrix}
-  \begin{pmatrix}
p\left(\dfrac{\pt R_1}{\pt \z'_{1}}\right) & \cdots  &
p\left(\dfrac{\pt R_1}{\pt \z'_{N}}\right) \\
\vdots & & \vdots \\
p\left(\dfrac{\pt R_N}{\pt \z'_1}\right) & \cdots  &
p\left(\dfrac{\pt R_N}{\pt \z'_{N}}\right)
\end{pmatrix} 
\ee
up to sign where $f_0$ and $f_1$ are $Z$-adjacent faces of $\Delta_j$.
\end{proposition}

\begin{proof} 
Two faces of $\Delta_j$ are $Z$-adjacent if and only if they are adjacent to either 
the edge $(01)$ or $(23)$. We first consider two faces adjacent to the edge $(01)$.
One of the two faces is oriented inward to $\Delta_j$, and the other is oriented
outward. Let $f_0$ and $f_1$ be the former and the latter, respectively, as in
Figure~\ref{fig.smoothing_b}. Note that the orientation of every edge of $f_0$ and
$f_1$ is determined, regardless of the vertices 2 and 3 of $\Delta_j$.
\begin{figure}[htpb!]
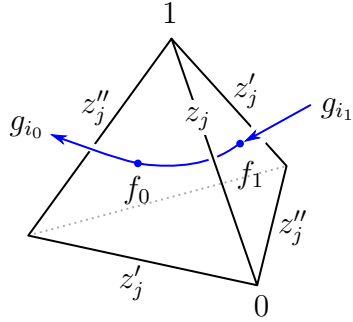
\caption{Two generators joined by B-smoothing.}
\label{fig.smoothing_b}
\end{figure}

From the edges of $f_0$ and $f_1$, we deduce that the generators $g_{i_0}$ and
$g_{i_1}$ corresponding to $f_0$ and $f_1$ respectively appear in the words
$R_1,\ldots,R_N$ as follows. 
\be
\label{eqn.appear2}
\begin{array}{l}
\cdots g_{i_1}  \z_j  \, g_{i_0} \cdots \\[3pt]
\cdots  \z_j '' \, g_{i_0}  \cdots \\[3pt]
\cdots   g_{i_0}^{-1}\,\z_j' \cdots \\[3pt]
\cdots g_{i_1}\,  \z_j '  \cdots  \\[3pt]
\cdots  \z_j'' \,  g_{i_1}^{-1} \cdots 
\end{array}
\ee
We stress that $g_{i_0}$ and $g_{i_1}$ do not appear elsewhere other than listed
above, and neither do $\z'_j$ and $\z''_j$. It follows that for all $1 \leq k \leq N$ 
\be 
\label{eqn.key}
p \left( \frac{\pt r_k}{ \pt g_{i_0}}- \frac{\pt r_k}{ \pt g_{i_1}} \, g_{i_1} \right)
=  p\left(\frac{\pt R_k}{\pt \z_j''} - \frac{\pt R_k}{\pt \z_j'} \right) \,.
\ee
Writing the above equation in a matrix form, we obtain the proposition. We prove
similarly for two faces adjacent to the edge $(23)$, in which case the left-hand side
of~\eqref{eqn.key} is equal to negative of the right-hand side.
\end{proof}

\begin{remark}
\label{rmk.a}
One can deduce similar equations for $Z'$ and $Z''$-adjacent faces, but the
equations only hold modulo 2. Precisely,  for $Z'$-adjacent faces $f_0$ and
$f_1$  of $\Delta_j$, the column vector~\eqref{eqn.vec}  and the $j$-th column of
\be
\begin{pmatrix}
p\left(\dfrac{\pt R_1}{\pt \z_{1}}\right) & \cdots  &
p\left(\dfrac{\pt R_1}{\pt \z_{N}}\right) \\
\vdots & & \vdots \\
p\left(\dfrac{\pt R_N}{\pt \z_1}\right) & \cdots  &
p\left(\dfrac{\pt R_N}{\pt \z_{N}}\right)
\end{pmatrix} -
\begin{pmatrix}
p\left(\dfrac{\pt R_1}{\pt \z''_{1}}\right) & \cdots  &
p\left(\dfrac{\pt R_1}{\pt \z''_{N}}\right) \\
\vdots & & \vdots \\
p\left(\dfrac{\pt R_N}{\pt \z''_1}\right) & \cdots  &
p\left(\dfrac{\pt R_N}{\pt \z''_{N}}\right) 
\end{pmatrix}
\ee
are congruent modulo 2, i.e. they induce the same vector over $(\BZ/2\BZ)[\pi]$.
Similarly, for $Z''$-adjacent faces $f_0$ and $f_1$ of $\Delta_j$, the column
vector~\eqref{eqn.vec}  and the $j$-th column of
\be
\begin{pmatrix}
p\left(\dfrac{\pt R_1}{\pt \z_{1}}\right) & \cdots  &
p\left(\dfrac{\pt R_1}{\pt \z_{N}}\right) \\
\vdots & & \vdots \\
p\left(\dfrac{\pt R_N}{\pt \z_1}\right) & \cdots  &
p\left(\dfrac{\pt R_N}{\pt \z_{N}}\right)
\end{pmatrix} -
\begin{pmatrix}
p\left(\dfrac{\pt R_1}{\pt \z'_{1}}\right) & \cdots  &
p\left(\dfrac{\pt R_1}{\pt \z'_{N}}\right) \\
\vdots & & \vdots \\
p\left(\dfrac{\pt R_N}{\pt \z'_1}\right) & \cdots  &
p\left(\dfrac{\pt R_N}{\pt \z'_{N}}\right) 
\end{pmatrix}
\ee
are congruent modulo 2
\end{remark}

\subsection{Determinants of NZ matrices and the (twisted) Alexander
polynomial}
\label{sub.thm13}

Combining Propositions~\ref{prop.fox}--\ref{prop.vec}, we obtain   
Theorems~\ref{thm.1} and ~\ref{thm.2}. We present details here.
 
\begin{proof}[Proof of Theorems~\ref{thm.1} and~\ref{thm.2}.]
Let $\calD$ be the dual cell complex of $\calT$ and consider the cellular chain 
complex of $\calD$ with local coefficient $\BZ[t^{\pm1}]$ twisted by $\alpha : \pi
\rightarrow \BZ \simeq t^\BZ$:
\be
\label{eqn.chain}
0 \longrightarrow C_2(\calD; \BZ[t^{\pm1}]_\alpha)
\overset{\partial_2}{\longrightarrow}
C_1(\calD; \BZ[t^{\pm1}]_\alpha)  \overset{\partial_1}{\longrightarrow}
C_0(\calD; \BZ[t^{\pm1}]_\alpha) \longrightarrow 0 \,.
\ee
Here $C_i(\calD ; \BZ[t^{\pm1}]_\alpha)  := C_i(\ti \calD ; \BZ) \otimes _{\BZ[\pi]}
\BZ[t^{\pm1}]$, where $\ti \calD$ is the universal cover of $\calD$, is a free
$\BZ[t^{\pm1}]$-module of rank $N$ 
for $i=0,2$ and of rank $2N$ for $i=1$. 

We choose a spanning tree of $\calD$, hence $N-1$ edges of $\calD$. Lifting the
tree to $\ti \calD$, we  obtain a basis of $C_i(\calD;\BZ[t^{\pm1}]_\alpha)$.
It is well-known that the boundary map $\pt_2$ in~\eqref{eqn.chain} is given
by the Fox derivative
\be
\pt_2=
\begin{pmatrix}
\alpha(p(\frac{\pt r_1}{\pt g_{1}})) & \cdots  & \alpha(p(\frac{\pt r_1}{\pt g_{2N}}))
\\
\vdots & & \vdots
\\
\alpha(p(\frac{\pt r_N}{\pt g_1})) & \cdots  & \alpha(p(\frac{\pt r_N}{\pt g_{2N}}))
\\
\end{pmatrix}^T \in M_{2N \times N}(\BZ[t^{\pm1}]) 
\ee
where $p$ is the map eliminating all generators in the tree. Also, the boundary map
$\pt_1$ can be expressed in terms of the vector described in~\eqref{eqn.v}.
Precisely, the $j$-th row of $\pt_1$ is  
\be
\alpha(v_{f_0}^T)+\cdots + \alpha(v_{f_3}^T) \in \BZ[t^{\pm1}]^{2N}
\ee
where $f_0,\ldots, f_3$ are the faces of $\Delta_j$. Recall that $v_f$ is a column
vector, hence its transpose $v_f^T$ is a row vector.
Since $Z$-smoothing couples the faces $f_0,\ldots,f_3$ of $\Delta_j$ into pairs,
we can decompose $\pt_1$ into
\be
\pt_1 = \pt_{1,B} +(\pt_1-\pt_{1,B})
\ee
where the $j$-th rows of both $\pt_{1,B}$ and $\pt_1 - \pt_{1,B}$ are of the form 
$\alpha(v_{f}^T)+\alpha(v_{f'}^T)$ for $Z$-adjacent faces $f$ and $f'$ of $\Delta_j$.
Then Propositions~\ref{prop.fox} and \ref{prop.vec} imply that
\be 
\label{eqn.rephrase}
\pt_2^T \, \pt_{1,B}^T = D \mathbf{B}_\alpha(t) 
\ee
where $D$ is a diagonal matrix with entries in $\{ \pm t^k \, | \, k \in \BZ \}$.
It follows that for any $N$-tuple $b=(b_1,\ldots,b_N)$ of column vectors in
$C_1(\calD;\BZ[t^{\pm1}]_\alpha)$, we have
\be
\begin{pmatrix}
\ \pt_2 ^T\ \\
\hline
\  b^T \ 
\end{pmatrix} 
\begin{pmatrix}
\, \pt_{1,B}^T  &  \rvline & \pt_1^T \,
\end{pmatrix} 
=
\begin{pmatrix}
D \mathbf{B}_\alpha(t) & \rvline & 0 \\
\hline
\pt_{1,B}(b)^T 	& \rvline & \pt_1(b)^T 
\end{pmatrix}
\ee
and thus 
\be
\label{eqn.comp}
\frac{\det \begin{pmatrix} 
\, \pt_2 & \rvline &  b \,
\end{pmatrix}}
{\det \pt_1(b)} \det\begin{pmatrix}
\, \pt_{1,B}\, \\
\hline
\pt_1 
\end{pmatrix} 
\doteq \det \mathbf{B}_\alpha(t)
\ee
provided that $\det \pt_1(b) \neq 0$.

The first term of the left-hand side of~\eqref{eqn.comp} is by definition
$\Delta_\alpha(t)/(t-1)$ where $\Delta_\alpha(t)$ is the Alexander polynomial
associated with $\alpha$. The second term obviously satisfies 
\be
\label{eqn.fin}
\det\begin{pmatrix}
\, \pt_{1,B}\, \\
\hline
\pt_1 
\end{pmatrix} = \det\begin{pmatrix}
\, \pt_{1,B}\, \\
\hline
\pt_1 - \pt_{1,B}
\end{pmatrix} \, .
\ee
Recall that each row of $\pt_{1,B}$ and $\pt_1 - \pt_{1,B}$ is of the form 
$v_{f}^T+v_{f'}^T$ for some faces $f$ and $f'$ and that each column of $\pt_1$
has at most two non-trivial entries. It follows that each row and column of the
matrix in the right-hand side of~\eqref{eqn.fin} has at most two non-trivial entries.
Such a matrix after changing some rows and columns can be expressed as a direct sum 
of matrices of the form 
\be
\label{eqn.shift}
\begin{pmatrix}
x_1 &  -y_1&  &  \\
& x_2 & -y_2 &    \\
& & \ddots & \ddots  \\
&  &  & x_{n-1} & -y_{n-1} \\
-y_n &  &  & & x_n
\end{pmatrix} 
\ee
whose determinant is $x_1 \cdots x_n - y_1 \cdots y_n$. In our case, expressing 
the matrix in the right-hand side of~\eqref{eqn.fin} as in the form~\eqref{eqn.shift}
is carried out by following the $Z$-curves. In particular, all $x_i$ are of the form
$t^{\alpha(g_i)}$ and all $y_i$ are $1$. It follows that the right-hand side
of~\eqref{eqn.fin} equals to $\prod (t^{\alpha(Z_i)}-1)$ where the product is
over all components $Z_i$ of the $Z$-curves. Therefore, we obtain
\be
\label{eqn.B2}
\det \mathbf{B}_\alpha(t) \doteq \frac{\Delta_\alpha(t)}{t-1} \, \prod_i
(t^{\alpha(Z_i)} - 1) \, .
\ee
On the other hand, Proposition~\ref{prop.Bloop} says that the $Z$-curves
homotope to disjoint peripheral curves. If one component is homotopically trivial,
we have $\det \mathbf{B}_\alpha(t) =0$ from Equation~\eqref{eqn.B2}.
Otherwise, the $Z$-curves are $m$-parallel copies of a peripheral curve $\gamma$
for $m \geq 1$, hence Equation~\eqref{eqn.Bsim} holds for $n=\alpha(\gamma)$.
This completes the proof of Theorem~\ref{thm.1}.

We obtain Theorem~\ref{thm.2} by simply replacing $\alpha$ in the above proof 
of Theorem~\ref{thm.1}
by $\alpha \otimes \rho$. We omit details, as this is indeed a repetition with only
obvious variants. For instance, the coefficient of the chain complex
\eqref{eqn.chain} is replaced by 
$(\BZ[t^{\pm1}] \otimes \BC^n)_{\alpha \otimes \rho}$,  the
matrix~\eqref{eqn.shift} should be viewed as a block matrix, and 
Equation~\eqref{eqn.B2} is replaced by
\be
\label{eqn.B4}
\det \mathbf{B}_{\alpha \otimes \rho} (t) \doteq \Delta_{\alpha \otimes \rho}(t) \,
\prod_{i} \det(\rho(Z_i) \, t^{\alpha(Z_i)} - I_n) 
\ee
where $I_n$ is the identity matrix of rank $n$.
\end{proof}

\begin{remark}
\label{rmk.A}
Applying the same argument as in the proof of Theorem~\ref{thm.1}, 
we deduce equations in $(\BZ/2\BZ)[t^{\pm1}]$ from Remark~\ref{rmk.a},
analogous to Equation~\eqref{eqn.B2}:
\begin{align}
\det \mathbf{A}_\alpha(t)&\equiv \frac{\Delta_\alpha(t)}{t-1} \, 
\prod_{i} (t^{\alpha(Z''_i)} - 1)  & (\textrm{mod } 2)\, , \\
\det ( \mathbf{A}_\alpha(t)-\mathbf{B}_\alpha(t))&\equiv \frac{\Delta_\alpha(t)}{t-1} \, 
\prod_{i} (t^{\alpha(Z'_i)} - 1)  & (\textrm{mod } 2) \,.
\end{align}
Here $Z'_i$ and $Z''_i$ are the $Z'$ and $Z''$-curves of $\calT$, respectively.
These equations usually fail in $\BZ[t^{\pm1}]$; see Section~\ref{sec.example} for
an example.
\end{remark}

\begin{remark}
\label{rmk.palin}
The palindromicity of $\det \mathbf{B}_\alpha(t)$ follows from the palindromicity
of the Alexander polynomial $\Delta_\alpha(t)$~ \cite{Milnor} together
with Theorem~\ref{thm.1}. Here we call a Laurent polynomial $p(t)$ palindromic 
if $p(t) \doteq p(t^{-1})$.
Equation~\eqref{ABt} specialized to~\eqref{ABa} implies
that if $\mathbf{B}_\alpha(t)$ is non-singular, then 
$\mathbf{B}_\alpha(t)^{-1} \mathbf{A}_\alpha(t)$ is invariant under the transpose
followed the involution $t \mapsto t^{-1}$. Hence $\det \mathbf{B}_\alpha(t)^{-1}
\mathbf{A}_\alpha(t)$ is palindromic, and so is $\det \mathbf{A}_\alpha(t)$.	
\end{remark}
\subsection{FK determinants of NZ matrices and the $L^2$-Alexander
	torsion}
\label{sub.thm3}

Imitating the proof of Theorem~\ref{thm.1} with the Fuglede-Kadison determinant,
we obtain Theorem~\ref{thm.3}. We present details here.

\begin{proof}[Proof of Theorem~\ref{thm.3}.]
Let $\calD$ be the dual cell complex of $\calT$. The universal cover  $\ti \calD$
of $\calD$ has the cellular chain complex of left $\BZ[\pi]$-modules
\be
0 \longrightarrow C_2(\widetilde{\calD};\BZ)
\overset{\partial_2}{\longrightarrow}
C_1(\widetilde{\calD};\BZ)  \overset{\partial_1}{\longrightarrow}
C_0(\widetilde{\calD};\BZ) \longrightarrow 0 
\ee
where $C_i:=C_i(\widetilde{\calD};\BZ)$  has rank $N$ for $i=0,2$ and rank
$2N$ for $i=1$.  The boundary maps $\pt_i : C_i \rightarrow C_{i-1}$ act on the
right, i.e., we have
\be
\partial_2 \in M_{N,2N}(\BZ[\pi]), \quad \partial_1 \in M_{N,2N}(\BZ[\pi]) \, .	
\ee
As in the proof of Theorem~\ref{thm.1}, we decompose $\partial_1$ as 
$\pt_1 = \pt_{1,B} +(\pt_1-\pt_{1,B})$
where the $j$-th columns of both $\pt_{1,B}$ and $\pt_1 - \pt_{1,B}$ are of the form 
$v_{f}+ v_{f'}$ for $Z$-adjacent faces $f$ and $f'$ of $\Delta_j$.
Then Propositions~\ref{prop.fox} and \ref{prop.vec} imply that
\be 
\label{eqn.key2}
 \pt_2 \, \pt_{1,B}  =  D \mathbf{B}
\ee
where $D$ is a diagonal matrix with entries in $\pm \pi$.
 
We now fix $t \in \BR^{+}$ and twist the coefficient of $C_i$ by using the
homomorphism $\alpha_t$, i.e. consider the chain complex
$C'_i := \BR[\pi] \otimes_{\BZ[\pi]} C_i$ where $\BR[\pi]$ is viewed as a
$\BZ[\pi]$-module using the homomorphism $\alpha_t$.
Note that the boundary maps of $C'_i$ are given by $\pt'_i = \alpha_t(\pt_i)$.
It follows from Equation~\eqref{eqn.key2} that for
any $N$-tuple $b=(b_1,\ldots,b_N)$ of (row) vectors, we have
\be
\begin{pmatrix}
\ \pt'_2\ \\
\hline
\  b  \ 
\end{pmatrix} 
\begin{pmatrix}
\, \pt'_{1,B}   &  \rvline & \pt'_1 \,
\end{pmatrix} 
=
\begin{pmatrix}
\alpha_t(D \mathbf{B}) & \rvline & 0 \\
\hline
\pt'_{1,B}(b) 	& \rvline & \pt'_1(b)
\end{pmatrix}
\ee
where $\pt'_{1,B} = \alpha_t(\pt_{1,B})$. Therefore, we obtain
\be
\label{eqn.comp2}
\frac{\rdet \begin{pmatrix}
\ \pt'_2\ \\
\hline
\  b  \ 
\end{pmatrix}} {\rdet( \pt_1(b))}\, 
\rdet\begin{pmatrix}
\, \pt'_{1,B}   &  \rvline & \pt'_1 \,
\end{pmatrix}  
=t^k \,  \rdet (\alpha_t(\mathbf{B}))
\ee
for fixed $k \in \BZ$,
provided that $\rdet (\pt_1(b)) \neq 0$. The first term of the left-hand side
of~\eqref{eqn.comp2} is $\tau^{(2)}(M,\alpha)(t)$ (see \cite[Lemma~3.1]{DFL15}), and
the second term satisfies
\be
\label{eqn.dd}
\rdet\begin{pmatrix}
\, \pt'_{1,B}   &  \rvline & \pt'_1 \,
\end{pmatrix}  
=
\rdet\begin{pmatrix}
\, \pt'_{1,B}   &  \rvline & \pt'_1 - \pt'_{1,B} \,
\end{pmatrix}  \, .
\ee
Recall that the matrix $( \pt_{1,B} \, | \,  \pt_1 - \pt_{1,B} )$ after changing some
rows and columns is the direct sum of matrices of the form~\eqref{eqn.shift} with
$x_i \in \pi$ and $y_i =1$. Such matrices decompose into
\begin{equation*}
\begin{pmatrix}
x_1 &  &  &  \\
& x_2 & &    \\
& & \ddots &   \\
&  &  & x_{n-1} & \\
&  &  & & x_n
\end{pmatrix} 
\begin{pmatrix}
1 &  -x_1^{-1} &  &  \\
& 1 & - x_2^{-1}&    \\
& & \ddots & \ddots  \\
&  &  & 1 & - x_{n-1}^{-1}\\
&  &  & & 1
\end{pmatrix} 
\begin{pmatrix}
1 - x_1^{-1} \cdots x_n^{-1}&  &  &  \\
- x_2^{-1} \cdots x_n^{-1}& 1 &  &    \\
\vdots & & \ddots &  \\
-x_{n-1}^{-1}x_n^{-1} &  &  & 1 & \\
- x_n^{-1} &  &  & & 1
\end{pmatrix},
\end{equation*}
hence we deduce that
\be
\rdet\begin{pmatrix}
\, \pt'_{1,B}   &  \rvline & \pt'_1 - \pt'_{1,B} \,
\end{pmatrix} =\prod_i \rdet (1-\alpha_t(Z_i^{-1}))
=\prod_i \rdet (1- t^{-\alpha(Z_i)} Z_i^{-1}) 
\ee  
where the products are over all the components $Z_i$ of the $Z$-curves of $\calT$.
Since we assumed that each component $Z_i$ has infinite order
in $\pi$, $\rdet(1 - t^{-\alpha(Z_i)} Z_i^{-1})$ is the Mahler measure of
$Z_i - t^{-\alpha(Z_i)}$, viewed as a polynoimal in $Z_i$, which equals to
$\mathrm{max} \{1, t^{-\alpha(Z_i)}\}$. It follows that
\be
\label{eqn.ll}
 \rdet (\alpha_t(\mathbf{B}))
=t^{-k} \, \tau^{(2)} (M,\alpha)(t) \, \prod_i \mathrm{max} \{1, t^{-\alpha(Z_i)}\} 
\ee
for fixed $k \in \BZ$. Since each component $Z_i$ is of infinite order and,
in particular, non-trivial, Proposition~\ref{prop.Bloop} implies that all
$\alpha(Z_i)$ should be the same up to sign. Thus Equation~\eqref{eqn.ll}
 implies Theorem~\ref{thm.3}.
\end{proof}

\section{Example}
\label{sec.example}

As is customary in hyperbolic geometry, in this section we give an example
of a cusped hyperbolic 3-manifold $M$, the complement of the knot $4_1$ in $S^3$.
The default \texttt{SnapPy} triangulation $\calT$ of $M$ consists of two ideal
tetrahedra $\Delta_1$ and $\Delta_2$, and is orderable with the ordering
shown in Figure~\ref{fig.Example41}~\cite{snappy}. It has two edges $e_1$ and $e_2$; 
$(01), (03), (23)$ of $\Delta_1$ and $(02), (12), (13)$ of $\Delta_2$ are identified
with $e_1$; $(02), (12), (13)$ of $\Delta_1$ and $(01),(03),(23)$ of $\Delta_2$ are
identified swith $e_2$.

\begin{figure}[htpb!]
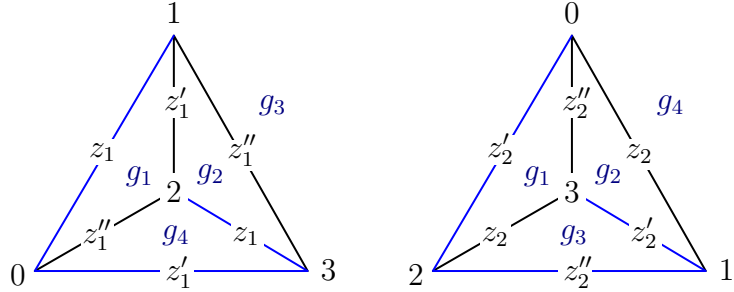
\caption{An ordered ideal triangulation of $4_1$.}
\label{fig.Example41}
\end{figure}


The dual cell complex of $\calT$ has $4$ edges and $2$ faces, hence we have two
words $r_1$ and $r_2$ in four generators $g_1,\ldots, g_4$
Note that  $g_1$ and $g_4$ (resp., $g_2$ and $g_3$) are oriented inward to $\Delta_1$
(resp., $\Delta_2$) and that the words $r_1$ and $r_2$ are obtained from winding
around the edges of $\calT$:
\be
\label{eqn.41}
\begin{aligned}
e_1 & : \quad
g_1 \overset{z_1}{\longrightarrow}  g_3
\overset{z''_2}{\longrightarrow} g_4
\overset{z_1}{\longrightarrow} g_2
\overset{z'_2}{\longrightarrow}
g_3^{-1} \overset{z'_1}{\longrightarrow}
g_4^{-1} \overset{z'_2}{\longrightarrow} g_1 \, ,\\
e_2 & : \quad
g_1 \overset{z_1'}{\longrightarrow}  g_2
\overset{z_2}{\longrightarrow} g_4
\overset{z_1''}{\longrightarrow} g_1^{-1}
\overset{z''_2}{\longrightarrow} g_2^{-1}
\overset{z''_1}{\longrightarrow} g_3 \overset{z_2}{\longrightarrow} g_1 \, .
\end{aligned}
\ee
Precisely, $r_1= g_3 g_4 g_2 g_3^{-1} g_4^{-1} g_1$ and
$r_2= g_2 g_4 g_1^{-1} g_2^{-1} g_3 g_1$.
Eliminating one generator, say $g_1$, we obtain a presentation of
$\pi=\pi_1(M)$:
\be
\pi = \langle g_{2}, g_3, g_4 \, | \, g_3 g_4 g_2 g_3^{-1} g_4^{-1} , 
\ g_2 g_4  g_2^{-1} g_3 \rangle \,.
\ee 
Note that $g_4$ is a meridian of the knot.

As in Section~\ref{sec.fox}, we define a word $R_i$ for
$i=1,2$ by inserting shape parameters to the word $r_i$ (c.f. \eqref{eqn.41}):
\begin{align*}
R_1 &=\z_1\, g_3\, \z''_2 \, g_4 \, \z_1 \, g_2 \, \z'_2 \, g_3^{-1}
\, \z'_1 \, g_4^{-1} \, \z'_2 \, g_1  \, ,\\
R_2&= \z'_1 \, g_2 \, \z_2 \, g_4 \, \z''_1 \, g_1^{-1}
\, \z''_2 \, g_2^{-1} \, \z''_1 \, g_3 \, \z_2 \, g_1 \, .
\end{align*}
Due to Proposition~\ref{prop.fox}, the twisted gluing equation matrices
$\mathbf{G}^\square$ of $\calT$ are equal to 
$(\pt R_i / \pt \z_j^\square)$ followed by eliminating $g_1$ and all $\z_j^\square$.
Explicitly, we have
\begin{align*}
\mathbf{G}
& = \begin{pmatrix}
1 + g_3 g_4 &  0  \\
0 & g_2 + g_2 g_4 g_2^{-1} g_3
\end{pmatrix}, \\
\mathbf{G}'
& = \begin{pmatrix}
g_3 g_4 g_2 g_3^{-1}& g_3 g_4 g_2 + g_3 g_4 g_2 g_3^{-1} g_4^{-1}  \\
1 & 0
\end{pmatrix}, \\
\mathbf{G}''
&= \begin{pmatrix}
0 & g_3 \\
g_2 g_4 + g_2 g_4  g_2^{-1}  & g_2 g_4
\end{pmatrix}   
\end{align*}
and thus the twisted Neumann--Zagier matrices of $\calT$ are given as
\begin{align}
\mathbf{A}
& = \begin{pmatrix}
1 + g_3 g_4 - g_3 g_4 g_2 g_3^{-1} & - g_3 g_4 g_2 - g_3 g_4 g_2 g_3^{-1} g_4^{-1}  \\
-1 & g_2 + g_2 g_4 g_2^{-1} g_3
\end{pmatrix}, \\
\mathbf{B}
&= \begin{pmatrix}
- g_3 g_4 g_2 g_3^{-1} &  g_3 - g_3 g_4 g_2 - g_3 g_4 g_2 g_3^{-1} g_4^{-1}  \\
g_2 g_4 + g_2 g_4 g_2^{-1}  -1 & g_2 g_4
\end{pmatrix} \, . \label{eqn.B41}
\end{align}

On the other hand, the abelianization map $\alpha:\pi\rightarrow \BZ$ is given by  
$\alpha(g_2)=0$, $\alpha(g_3)= {-1}$ and $\alpha(g_4) = {1}$. Applying $\alpha$ to
the twisted Neumann--Zagier matrices, we obtain
\be
\mathbf{A}_\alpha(t)
= \begin{pmatrix}
2- t & -2 \\
-1 &  2
\end{pmatrix}, \qquad
\mathbf{B}_\alpha(t)
= \begin{pmatrix}
- t & t^{-1} -2\\
2 t  -1 & t
\end{pmatrix} \, .
\ee
One easily computes that 
\be
\det \mathbf{B}_\alpha(t) \doteq (t-1) (t^2-3t+1)
\ee
which verifies Theorem~\ref{thm.1} as well as Equation~\eqref{eqn.B2}. Note that
the Alexander polynomial of $4_1$ is $t^2-3t+1$ and that $\calT$ has two $Z$-curves 
$Z_1 =  g_1 g_3$ and $Z_2 = g_4 g_2$  with $ \alpha(Z_1)=-1$, $\alpha(Z_2) =1$.
We also check that  
\begin{equation}
	\det \mathbf{A}_\alpha(t) \equiv \det (\mathbf{A}_\alpha(t)-\mathbf{B}_\alpha(t))
\equiv 0 \quad \textrm{ (mod 2) }
\end{equation}
which verifies Remark~\ref{rmk.A}. Note that $\calT$ has 
one $Z'$-curve  $Z'_1 = g_1 g_2 g_3^{-1} g_4^{-1}$ with $\alpha(Z'_1) = 0$ and 
one $Z''$-curve $Z''_1 = g_2^{-1} g_3 g_4 g_1^{-1}$ with $\alpha(Z''_1) = 0$.

We now compute a (positive) lift $\rho : \pi \rightarrow \SL_2(\BC)$ of the geometric
representation of $M$. Since $g_4$ is a meridian of the knot, we may let (see
\cite[Lemma 1]{Riley})
\be
\rho(g_4) = \begin{pmatrix}
1 & 1 \\ 0& 1
\end{pmatrix}\, , \quad 
\rho(g_2) = \begin{pmatrix}
n & 0 \\ u & 1/n
\end{pmatrix} \,.
\ee
A straightforward computation shows that the above assignment induces a
representation $\rho$ of $\pi$ if and only if  $u = -(1 -4 n^2 +n^4)/(3 n + 3n^3)$
and $1 - 3n + 5n^2 - 3n^3 + n^4=0$. Applying $\alpha \otimes \rho$ to 
Equation~\eqref{eqn.B41}, one computes that 
\be
\det \mathbf{B}_{\alpha \otimes \rho}(t) \doteq (t-1)^4 (t^2-4t+1)/t^2 \, .
\ee
This verifies Theorem~\ref{thm.2} as well as Equation~\eqref{eqn.B4}. Note that
the twisted Alexander polynomial of $4_1$ associated with $\alpha \otimes \rho$
is $t^2-4t+1$. 

\begin{remark}
\label{rmk.82}
For ordered ideal triangulations, $\det \mathbf{A}_\alpha(t)$ is often a multiple of 2
and thus vanishes in $(\BZ/2\BZ)[t^{\pm1}]$. One example which is not the case is the
knot $8_2$. Its default \texttt{SnapPy} triangulation is orderable, and Philip Choi's
program computes that 
\begin{align*}
\mathbf{A}_\alpha(t) &=
\left(
\begin{array}{cccccc}
t^{-4}+1 & 1 & -t^{-4} & t^{-4} & 0 & 0 \\
-t^{-2}-1 & 0 & 1 & 0 & 0 & 0 \\
0 & -1 & 0 & t & 1 & -1 \\
0 & -t^{-1} & -t & -t^{-4} & -t & t^{-5} \\
0 & 1 & 0 & 0 & 0 & t \\
0 & 0 & 1 & -1 & 0 & -1 \\
\end{array}
\right), \\
\mathbf{B}_\alpha(t)&=
\left(
\begin{array}{cccccc}
t^{-2} & 0 & -t^{-4} & 0 & 0 & 0 \\
-t^{-2}-1 & t^{-2} & 0 & 0 & 0 & 0 \\
1 & -1 & 1 & 0 & t & t-1 \\
0 & 1-t^{-1} & -t & t^{-5}-t^{-4} & -t & 0 \\
0 & 0 & t^2 & t^2 & -t & t \\
0 & 0 & 0 & -1 & 1 & -1 \\
\end{array}
\right) 
\end{align*}
with
\begin{align*}
\det \mathbf{A}_\alpha(t)& = (t-1) (t^{12}+t^7-2 t^6+t^5+1)  \, ,\\
\det \mathbf{B}_\alpha(t) & =(t-1) (t^6-3 t^5+3 t^4-3 t^3+3 t^2-3 t+1)  \, .
\end{align*}
Note that the Alexander polynomial of the knot $8_2$ is the second factor of 
$\det \mathbf{B}_\alpha(t)$ (hence this verifies Theorem~\ref{thm.1}) and that
\begin{align*}
\det \mathbf{A}_\alpha(t) & \equiv (t-1)^2 (t^4+t^3+t^2+t+1)
(t^6-3t^5+3t^4-3t^3+3t^2-3t+1) \quad (\textrm{mod } 2) \,.
\end{align*}
\end{remark}

\subsection*{Acknowledgments}

The authors wish to thank Thang Le for enlightening conversations. 
S.Y. wishes to thank Philip Choi for his help with computer calculations.


\bibliographystyle{hamsalpha}
\bibliography{biblio}
\end{document}